\documentclass[11pt,reqno]{amsart}
\usepackage{graphicx,amsmath,amssymb,amsfonts,amsthm,enumitem,bm,xcolor,mathrsfs}
\usepackage{mathtools, calc}
\usepackage{a4wide,fullpage}
\usepackage{cleveref}

\usepackage[normalem]{ulem} 
\usepackage{url}

\usepackage{color}
\usepackage{longtable}
\usepackage{comment}

\newcommand{\N}{\mathbb{N}}
\newcommand{\Z}{\mathbb{Z}}
\newcommand{\R}{\mathbb{R}}
\newcommand{\C}{\mathbb{C}}

\renewcommand{\t}{\tau}

\newcommand{\z}{\zeta}

\newcommand{\sgn}{\operatorname{sgn}}
\newcommand{\erf}{\operatorname{erf}}

\newcommand{\mat}[1]{\left(\begin{matrix}#1\end{matrix}\right)}

\renewcommand{\(}{\left\(}
\renewcommand{\)}{\right\)}

\numberwithin{equation}{section}
\theoremstyle{plain}
\newtheorem{theorem}{Theorem}[section]
\newtheorem{lemma}[theorem]{Lemma}

\newtheorem*{remark*}{Remark}

\newtheorem*{remarks*}{Remarks}

\newcommand{\lrb}[1]{\left(#1\right)}

\newtheorem{corollary}[theorem]{Corollary}
\newtheorem{proposition}[theorem]{Proposition}

\numberwithin{equation}{section}

\renewcommand{\binom}[2]{\left(\begin{smallmatrix}#1\\\\#2\end{smallmatrix}\right)}

\newcommand{\Pmod}[1]{\ ( \mathrm{mod} \, #1 )}
\newcommand{\Log}{\mathrm{Log}}

\setlist[enumerate]{leftmargin=*,label=\rm{(\arabic*)}}

\makeatletter
\@namedef{subjclassname@2020}{%
	\textup{2020} Mathematics Subject Classification}
\makeatother

\allowdisplaybreaks

\title{False and partial Eisenstein series\\ related to unimodal sequences}
\author[K. Bringmann]{Kathrin Bringmann}
\author[B. Pandey]{Badri Vishal Pandey}
\address{Department of Mathematics and Computer Science\\Division of Mathematics\\University of Cologne\\ Weyertal 86-90 \\ 50931 Cologne \\Germany}
\email{kbringma@math.uni-koeln.de}
\email{bpandey@uni-koeln.de, badrivishal9451@gmail.com}
\author{Jan-Willem van Ittersum}
\address{Department of Mathematics and Computer Science\\Division of Mathematics\\University of Cologne\\ Weyertal 86-90 \\ 50931 Cologne \\Germany}
\curraddr{Korteweg--de Vries Institute for Mathematics, University of Amsterdam, Postbus 94248, 1090 GE  Amsterdam, The Netherlands}
\email{j.w.m.vanittersum@uva.nl}

\makeatletter
\@namedef{subjclassname@2020}{%
	\textup{2020} Mathematics Subject Classification}
\makeatother


\begin{document}

\begin{abstract} 
	Motivated by the fact that the classical Jacobi theta function $\vartheta$ is the exponential generating function of the Eisenstein series, we study the exponential Taylor coefficients (in the elliptic variable) of a related natural partial theta function, as well as a false theta function related to the Dedekind eta function. We prove that the space spanned by these objects is closed under differentiation, analogous to the space of quasimodular forms, and that it contains the quasimodular forms themselves. We further provide their Fourier expansions, establish quasimodular completions, and derive a recursive formula for the Taylor coefficients of the logarithm of the unimodal rank generating function, expressed as partition traces of the false and partial objects.
\end{abstract}

\maketitle

\section{Introduction and statement of results}
The Jacobi theta function ($\z:=e^{2\pi i z}, q:=e^{2\pi i\tau}, z\in\C,\tau$ in the complex upper half plane $\mathbb H$) 
\begin{align}\label{def:Jacobi-theta-function}
	\vartheta(z;\tau):={i}\sum_{n\in\Z} (-1)^n q^{\frac{1}{2}\lrb{n+\frac12}^2}\z^{n+\frac12} = -iq^{\frac18} \z^{-\frac12} (q)_\infty (\z)_\infty (\z^{-1}q)_\infty,
\end{align}
where $(a)_n  := \prod_{j=0}^{n-1} (1-aq^j)$ for $n\in\N_0\cup\{\infty\}$,
and its generalizations serve as fundamental building blocks for elliptic functions and modular forms. Through the so-called theta decomposition they describe the structure of the space of Jacobi forms in terms of vector-valued modular forms. Furthermore, these functions appear in various areas, such as number theory, combinatorics, and mathematical physics, as well as partition theory and string theory (see \cite{GSW, Mumford, MumfordII, Nazaroglu13}).

The Taylor expansion of the classical Jacobi theta function is (up to a normalization) the exponential generating function of {\it Eisenstein series} \cite[(7)]{Zag91}, defined for weights $k\in\N$ by
\begin{align}\label{def:eisenstein}
	G_{k}(\tau):=\begin{cases}-\frac{B_k}{2k}+\sum_{n\ge1} \sigma_{k-1}(n)q^n & \text{if }2\mid k,\\
	0 & \text{otherwise,}\end{cases}
\end{align}
where $\sigma_{k-1}$ denotes the \emph{divisor function} $\sigma_{k-1}(n):=\sum_{d\mid n}d^{k-1}$. More precisely, we have 
\begin{align}\label{eq:vartheta-as-Eis}
	\vartheta(z;\tau)&=-2\pi z \eta^3(\tau)\exp\left(- 2\sum_{k\ge1} G_{k}(\tau)\frac{(2\pi iz)^k}{k!}\right).
\end{align}
In view of this, it is natural to study the exponential Taylor coefficients of false and partial theta functions introduced by Rogers \cite{Rog1917} and Ramanujan \cite{Ram1987}, respectively.
Define the partial and false theta functions $T$ and $h$ by\footnote{The definition of $h$ should be compared with the expansion of $\eta$ in the pentagonal number theorem. } 
\begin{align}
	T(z;\tau)&:=2i\sum_{n\ge0} (-1)^n \zeta^{n+\frac{1}{2}} q^{\frac{1}{2}\left(n+\frac{1}{2}\right)^2},\label{def:mathcalT}\\
	h(\z;q) &:= (1 - \zeta) \sum_{n\ge0} (-1)^n \zeta^{3n} q^{\frac{n(3n+1)}{2}}  \left(1 - \zeta^2 q^{2n+1}\right)\label{eq:def-G-H}\\
	&= \sum_{n\in\Z} (-1)^n \sgn \lrb{n+\frac{1}{2}} \left(\zeta^{\sgn\lrb{n+\frac{1}{2}}3n}-\zeta^{\sgn\lrb{n+\frac{1}{2}}(3n+1)}\right) q^{\frac{n(3n+1)}{2}}.\nonumber
\end{align}

Our aim is to study properties analogous to those of Eisenstein series for the exponential Taylor coefficients of $T$ and $h$.  That is, we investigate the sequences of functions $\{g_k\}_{k\in\N}$ and $\{h_k\}_{k\in\N}$, which we call {\it partial} and {\it false Eisenstein series}, respectively, defined by the identities,
\begin{align}
	T_0(\tau)\exp\left(- \sum_{k\ge 1}g_k(\tau)\frac{(2\pi iz)^k}{k!}\right) &:= {T(z;\tau)}, \qquad \left(T_0(\tau) := T(0;\tau)\right)\label{def:mathcalT-as-exp}\\
	-2i\sin(\pi z)q^{-\frac{1}{24}}\eta(\tau)\exp\left(-\sum_{k\ge 1}h_k(\tau)\frac{(2\pi iz)^k}{k!}\right)&:= h(\z;q).\label{def:mathcalH-as-exp}
\end{align}  
The functions $g_k$ and $h_k$ are neither (quasi)modular nor do they seem to admit (quasi)modular completions in the classical sense -- that is, the addition of a nontrivial function of $\tau$ and $\overline{\tau}$ to produce a modular object. Rather, we need to introduce an additional independent variable $w\in\mathbb{H}$ to define their completions. A function $\widehat{f}(\tau,\overline{\tau},w,\overline{w})$ is a (quasi)modular completion of $f$ if $\widehat{f}$ transforms (quasi)modularly under the simultaneous action $(\tau,w)\mapsto(\frac{a\tau+b}{c\tau+d},\,\frac{aw+b}{cw+d})$ and if $f$ is recovered from $\widehat{f}$ by taking an appropriate limit. In this paper, we have $f(\tau)=\lim_{t\to\infty}\widehat{f}(\tau,\overline{\tau},\tau+it+\varepsilon,\overline{\tau}-it+\varepsilon)$ for any $\varepsilon>0$ (see Subsection~\ref{sub:false-and-partial}). To state our result, for $\tau,w\in \mathbb{H}$ with $\tau\neq w$ and $\begin{psmallmatrix} a&b\\ c&d \end{psmallmatrix}\in\mathrm{SL}_2(\Z)$, define
\begin{equation}\label{def:chi-tau-w}
\chi_{\tau,w} \mat{a&b\\c&d} := \sqrt{\frac{i(w-\tau)}{(c\tau+d)(cw+d)}} \frac{\sqrt{c\tau+d}\sqrt{cw+d}}{\sqrt{i(w-\tau)}}\in\{1,-1\}.
\end{equation}
We let $\delta_{j,\ell}$ denote the Kronecker delta, i.e., $\delta_{j,\ell}=1$ if $j=\ell$ and $\delta_{j,\ell}=0$ otherwise.
\begin{theorem}\label{thm:comp-of-g_k}
	There exist functions $\widehat{g}_k(\tau,w)$ such that the following holds:
	\begin{enumerate}
		\item For $\gamma=\begin{psmallmatrix} a&b\\ c&d \end{psmallmatrix}\in\mathrm{SL}_2(\Z)$, we have
	\begin{align*}
		\widehat{g}_{k} \lrb{\frac{a\tau+b}{c\tau+d},\frac{aw+b}{cw+d}} &=  \chi^{k}_{\tau,w} (\gamma) (c\tau+d)^k   \widehat{g}_{k}(\tau,w) + \delta_{k,2}\frac{ic}{2\pi}(c\tau+d) .
	\end{align*}
	\item For any $\varepsilon>0$, we have
\begin{align*}
	\lim_{t\to\infty} \widehat{g}_{k}(\tau,\tau+it+\varepsilon) = g_k(\tau) .
\end{align*}
	\end{enumerate}
\end{theorem}
For $h_k$, we have the following result on their completions.
\begin{theorem}\label{thm:comp-of-h_k}
	There exist functions $\widehat{h}_k(\tau,w)$ such that the following holds:
	\begin{enumerate}
		\item For $\gamma=\begin{psmallmatrix} a&b\\ c&d \end{psmallmatrix}\in\Gamma_0(3)$, we have
	\begin{align*}
		\widehat{h}_{k} \lrb{\frac{a\tau+b}{c\tau+d},\frac{aw+b}{cw+d}} &= \chi^{k}_{\tau,w}  (\gamma) (c\tau+d)^k \widehat{h}_{k}(\tau,w)  +\delta_{k,2} \frac{3 i c}{2 \pi  }(c\tau +d).
	\end{align*}
	\item For any $\varepsilon>0$, we have
\begin{align*}
	\lim_{t\to\infty} \widehat{h}_{k}(\tau,\tau+it+\varepsilon) =h_k(\tau).
\end{align*}
	\end{enumerate}
\end{theorem}

Recall that the algebra generated over $\mathbb{C}$ by the Eisenstein series  is closed under differentiation (see \eqref{eq:der-of-Eis}). Recently, the authors introduced a new class of ``mock-Eisenstein'' series arising from the rank generating function and proved that these objects are likewise closed under differentiation (see \cite[Theorem~1.2 (3)]{BPvI25}). Rausch \cite{R25} then constructed a family of ``mock-Eisenstein'' series associated with so-called $k$-ranks, which also satisfy a closedness property for each $k\ge2$. In this paper, we show that the algebra generated by the functions $g_k$ (and $h_k$) enjoys the same property. More precisely, let  $\widetilde{\mathcal M}:=\C[G_2,G_4,G_6]$ be the space of quasimodular forms and
\[
\mathcal{T}:=\mathbb{C} \left[g_1,g_2,g_3,\ldots\right],
\quad \mathcal{H}:=\mathbb{C} \left[ h_1,h_2,h_3\ldots \right], \quad\text{and} \quad 
D:=q\frac{\partial}{\partial q}.
\]
\begin{theorem}\label{thm:mathbbT-algebra}
	The space $\mathcal{T}$ is closed under the action of $D$. Moreover, for $k\in \N$, we have
	\begin{equation*}
		D(g_k) = \frac{g_{k+2}}{2} - \frac12\sum_{d=0}^{k} \binom{k}{d}g_{d+1} g_{k-d+1},\qquad
		D\lrb{\Log \lrb{T_0}} = -\frac{g_2}{2} +\frac{g_1^2}{2}.
	\end{equation*}
\end{theorem}
Furthermore the space  $\mathcal{H}$ is a differential algebra extending the space of quasimodular forms:
\begin{theorem}\label{thm:h-n-diff}
	The space $\mathcal{H}$ is closed under the action of $D$. Moreover, for $k\in \N$, we have
	\begin{align*}
		D(h_k) &= \frac{h_{k+2}}{6} -\frac{1}{6}\sum_{d=0}^k \binom{k}{d} h_{d+1}h_{k-d+1},\quad
		\frac{h_2}{6} - \frac{h_1^2}{6} = G_2.
	\end{align*}
Further, we have $\mathcal{H} =\widetilde{\mathcal M}[h_1,h_3,h_5,\ldots]$. 
\end{theorem}
\begin{remarks*}\hspace{0cm} 	
\begin{enumerate}
	\item  By Theorem~\ref{thm:h-n-diff}, the space of quasimodular forms for $\mathrm{SL}_2(\Z)$ is contained in~$\mathcal{H}$, even though the completions of their generators~$h_k$ have quasi completions which transform on~$\Gamma_0(3)$ (see Theorem~{\rm \ref{thm:comp-of-h_k}}).
	\item  Numerical computations suggest that $\mathcal{T}$ is contained in the free algebra $\widetilde{M}[g_1,g_2,g_4,$ $g_6,\ldots]$. Moreover, it seems that $\mathcal{H}$ is freely generated by $G_2, G_4, G_6$, and $h_k$ for $k$ odd.
\end{enumerate}  
\end{remarks*}

Given the Eisenstein series like properties of the partial (resp.\ false) Eisenstein series $g_k$ (resp.\ $h_k$), it is natural to ask whether their Fourier expansions have a similar shape. We show that this is indeed the case. For this, we recall that the $n$-th Fourier coefficient of $G_{k}$, defined in \eqref{def:eisenstein}, is the sum of the $(k-1)$-th powers of the divisors of $n$. Here we show that the $n$-th Fourier coefficient of $g_k$ (resp.\ $h_k$) is an integer linear combination of $(k-1)$-th powers of integers in some specified range depending on $n$. 
Letting $\ell(\lambda):=\sum_{j=1}^n m_j$ denote the {\it length} of a partition $\lambda$, define the set\footnote{Throughout the paper, we use the notation $\lambda = (1^{m_1}, 2^{m_2}, \dots, n^{m_n}) \vdash n$ to denote a partition of $n$, where $m_j$ is the multiplicity of $j$.}
\begin{align*}
\Lambda(n,m):=\left\{ \lambda\vdash n \colon m_1\geq m_2\geq \ldots \geq m_n \geq 0,   \ell(\lambda)=m\right\}.
\end{align*}
\begin{theorem}\label{thm:Fourier-exp-of-g_k}
	For $k\in\N$, we have
	\begin{align*}
g_k(\tau) = {-\frac{\delta_{k,1}}{2}}+\sum_{n\ge1} \sum_{m=1}^n a_{n,m} {m^{k-1}} q^n,\quad \text{and} \quad T_0(\tau)= 2iq^{\frac{1}{8}}\exp\left( -\sum_{n\ge1} \sum_{m=1}^n \frac{a_{n,m}}{m} q^n\right).
	\end{align*}
	Here $a_{n,m}\in \Z$ and $a_{n,m}=0$ if $m<\lceil \tfrac{1}{2} (\sqrt{8n+1}-1) \rceil$. In fact, we have	
	\begin{align*}
		a_{n,m}:=\sum_{\lambda\in\Lambda(n,m)} (-1)^{m+m_1}\frac {m}{m_{1}}  \binom{m_1 }{m_1-m_2,m_2-m_3,\ldots}.
	\end{align*}
\end{theorem}

We have a similar result for $h_k$ with 
\begin{equation*}
\Omega(n,m):=\left\{ \lambda\vdash n \colon m_{3j-2}\ge m_{3j+1},     m_{3j-1}\ge m_{3j+2},     m_{3j}=0 \text{ for all }j\in\N,   3\ell(\lambda)=m+m_1 \right\}.
\end{equation*}

\begin{theorem}\label{thm:Fourier-exp-of-h_k}
	For $k\in\N$, we have
	\begin{align*}
		h_k(\tau) &= -\frac{\delta_{k,1}}{2} + \sum_{n\geq 1} \sum_{m=1}^{2n} b_{n,m} m^{k-1} q^n, \quad\text{and}\quad \eta(\tau)=q^{\frac{1}{24}}\exp\left(-\sum_{n\geq 1} \sum_{m=1}^{2n} \frac{b_{n,m}}{m} q^n\right)
	\end{align*}
with $b_{n,m}\in \Z$ and $b_{n,m}=0$ if $m< \lfloor \frac{1}{2} (\sqrt{24 n+1}-1)\rfloor$. In fact, we have 
	\begin{align*}
		b_{n,m}:=\sum_{\lambda\in\Omega(n,m)} (-1)^{m+m_2}\frac{m}{m_1+m_2}\binom{m_1+m_2}{m_2-m_5,m_5-m_8,\ldots m_1-m_4,m_4-m_7,\ldots}.
	\end{align*}
\end{theorem}

The algebra of functions $g_k$ and $h_k$ can be interpreted in terms of functions coming from unimodal sequences, defined as follows.
A \emph{unimodal sequence of size $n\in \N$} has the form\footnote{A sequence of integers, such as $(2,2)$, may give rise to multiple unimodal sequences; we use an underline on $c$ to denote that it is the peak of the unimodal sequence.}
\begin{align*}
	1\le a_1 \le \cdots \le a_{r} \le \underline{c} \ge b_{1} \ge \cdots \ge b_s\ge1,
\end{align*}
for integers $a_1,\ldots,a_r,c,b_1,\ldots,b_s$
with $r,s\in\N_0$ and $n=c+\sum_{j=1}^{r}a_j+\sum_{j=1}^{s} b_j$. Let $u(n)$ count the number of unimodal sequences of size $n$. The \emph{rank} of a unimodal sequence is $s-r$. By convention, the empty sequence is the unique unimodal sequence of size zero; it is of rank zero. Let $u(n,m)$ denote the number of unimodal sequences of size $n$ and rank $m$.
The generating function for the rank of unimodal sequences \cite[(2.2)]{KL} is given by ($(a,b)_n := (a)_n(b)_n$)
\begin{equation*}
	U(\zeta;q) := \sum_{\substack{n\ge0\\m\in\Z}} u(n,m) \zeta^m q^n = \sum_{n\ge0} \frac{q^n}{\left(\zeta q, \zeta^{-1}q\right)_n}.
\end{equation*}
We look at the exponential Taylor expansion of the unimodal ranks. Namely, we write
\begin{align}
	U(\zeta;q) &= : \frac{\sin(\pi z)}{\pi z}U(1;q) \exp \lrb{ 2\sum_{k\ge 1} u_k(\tau) \frac{(2\pi i z)^k}{k!}} . \label{eq:U-as-exp}
\end{align}

Motivated by the partition Eisenstein traces recently defined by Amdeberhan, Griffin, Ono, and Singh \cite{Traces}, the authors defined in \cite{BPvI25}, for a sequence of functions $f=\{f_k\}_{k\in\N}$ and $n\in\N_0$, the \emph{$n$-th partition trace} with respect to $f$ and a function~$\phi$ on partitions as 
\begin{equation*}
	\mathrm{Tr}_n(\phi,f;\t):=\sum_{\lambda\vdash n} \phi(\lambda)f_\lambda(\t),
\end{equation*}
where the sum ranges over all partitions of $n$ and, for $\lambda\vdash n$, we set
\begin{align*}
	f_\lambda(\t) := \prod_{j=1}^n f_j^{m_j}(\t).
\end{align*}
We show that $u_k$ can be recursively expressed as partition traces of the functions $g_k$ and~$h_k$. Let 
\[\mathcal{U}:=\mathcal{H}[g_1,g_2,g_4,g_6,\ldots].\]

\begin{theorem}\label{thm:recursive-formula-for-u-k}
	For $k\in \N$, we have $u_k\in\mathcal{U}$. In fact, we have 
	\begin{align}\label{eq:u_k}
	u_k(\tau)  &= -\frac{k!}{2} \sum_{\substack{\lambda\vdash k\\ \lambda\neq (k^1)}} \phi(\lambda) u_\lambda(\tau)+ \mathrm{Tr}_k(\phi,\gamma;\tau) + k! g_1(\tau)\mathrm{Tr}_{k-1}(\psi,h;\tau),
	\end{align}
	where $u:=\{u_k\}_{k\in\N},   \gamma:=\{G_k-2^{k-1}g_k\}_{k\in\N},   h:=\{h_k\}_{k\in\N}$, and, for $\lambda\vdash k$, we let
\begin{align*}
	\phi(\lambda) := \prod_{j=1}^k \frac{2^{m_j}}{j!^{m_j}m_j!},\qquad \psi(\lambda) := \prod_{j=1}^k \frac{(-1)^{m_j}}{j!^{m_j}m_j!}.
\end{align*}
Moreover, the space $\mathcal{U}$ is closed under the action of $D$.
\end{theorem}
\begin{remark*}
Note that $u_k=0$ if $k$ is odd, since $U(e^{2\pi i z};q)$ is an even function of $z$. Hence, for $k$ odd equation~\eqref{eq:u_k} gives a relation between $G_\ell, g_\ell$, and $h_\ell$ for $\ell\leq k$. 
\end{remark*}
As a direct corollary of Theorems~\ref{thm:comp-of-g_k}, \ref{thm:comp-of-h_k}, and \ref{thm:recursive-formula-for-u-k}, we recursively give a quasi-completion of $u_k$. 
 More precisely, define the sequence of functions $\widehat{u}:=\{\widehat{u}_k\}_{k\in\N}$ by
\begin{align*}
	\widehat{u}_k(\tau,w) &:= -\frac{k!}{2}\sum_{\substack{\lambda\vdash k\\ \lambda\neq (k^1)}} \phi(\lambda) \widehat{u}_\lambda(\tau,w)+\frac{k!}{2} \mathrm{Tr}_k\left(\phi,\widehat{\gamma};\tau,w\right) + k!   \widehat{g}_1(\tau,w)\mathrm{Tr}_{k-1}\left(\psi,\widehat{h};\tau,w\right),
	\end{align*}
	where $\widehat{\gamma}:=\{\widehat{G}_k-2^{k-1}\widehat{g}_k\}_{k\in\N},   \widehat{h}:=\{\widehat{h}_k\}_{k\in\N}$, and $\phi$ and $\psi$ are defined in Theorem~\ref{thm:recursive-formula-for-u-k}. 
\begin{corollary}\label{cor:u-k-hat} 
	The functions $\widehat{u}_k(\tau,w)$ transform quasimodular of weight $(k,0)$ for the simultaneous action of $\Gamma_0(3)$ on $\tau$ and $w$. Furthermore, for any $\varepsilon>0$, we have
	\begin{align*}
		\lim_{t\to\infty} \widehat{u}_k(\tau,\tau+it+\varepsilon)=u_k(\tau).
	\end{align*}
\end{corollary}

The paper is organized as follows. In Section~\ref{Sec:preliminaries}, we recall some combinatorial lemmas 
and some basic properties of the Dedekind eta function, (false) theta functions, and quasimodular forms. In Section~\ref{sec:False and partial Eisenstein series} we prove Theorems~\ref{thm:comp-of-g_k} and \ref{thm:comp-of-h_k}, in Section~\ref{sec:differential-eq}, Theorems~\ref{thm:mathbbT-algebra} and \ref{thm:h-n-diff}, in Section~\ref{sec:Fourier-exp}, Theorems~\ref{thm:Fourier-exp-of-g_k} and \ref{thm:Fourier-exp-of-h_k}, and in Section~\ref{sec:unimodal} Theorem~\ref{thm:recursive-formula-for-u-k}. In Section~\ref{sec:examples} we present some examples. Finally, in Section~\ref{sec:questions}, we raise some questions for future research.

\section*{Acknowledgements}
The first and the second author have received funding from the European Research Council (ERC) under the European Union's Horizon 2020 research and innovation programme (grant agreement No. 101001179), and the first and the third author are supported by the SFB/TRR 191 “Symplectic Structure in Geometry, Algebra and Dynamics”, funded by the DFG (Projektnummer 281071066 TRR 191).
We thank William Keith and Caner Nazoroglu for useful discussions.

\section{Preliminaries}\label{Sec:preliminaries}

\subsection{Combinatorial lemmas}
We require a lemma regarding the divisibility of multinomial coefficients, which follows directly from B\'ezout's Lemma.\footnote{This lemma was also used in \cite{BPvI25}.}
\begin{lemma}\label{lem:div}
	For $a_1,\ldots,a_\ell\in\N$ with $\sum_{j=1}^\ell a_j=n$, we have
	\[
	\frac{n}{\gcd(a_1,\ldots,a_\ell)}  \Big|  \binom{n}{a_1,a_2,\ldots,a_\ell}.
	\]
\end{lemma}

We also need the following elementary lemma.
\begin{lemma}\label{lem:ineq}
Let $n_1\geq n_2\geq \ldots \geq n_r\geq 0$, $a_1,\ldots,a_r, b_1,\ldots,b_r\in \R$, and assume that for all $s\in\{1,\ldots,r\},$ we have $\sum_{j=1}^s a_j \le \sum_{j=1}^s b_j$. Then
\[
\sum_{j=1}^r a_j n_j \leq \sum_{j=1}^r b_j n_j.
\]
\end{lemma}

Next we recall the Fa\`{a} di Bruno formula which gives ($n\in\N$)
\begin{align}
	\hspace{-0.25cm}{d^{n} \over dz^{n}}f(g(z)) =\sum_{\lambda \vdash n} {\frac {n!}{m_{1}!  m_{2}!  \cdots   m_{n}!}} f^{(m_{1}+\cdots +m_{n})}(g(z)) \prod _{j=1}^{n}\left({\frac {g^{(j)}(z)}{j!}}\right)^{m_{j}}.\label{eq:Fa-di-bruno}
\end{align}

Finally we use a result on Pólya’s cycle index polynomials \cite{S1999}.
\begin{lemma}[Example 5.2.10 of \cite{S1999}]\label{lem:cycle-index}
	We have, as a formal power series in $w$,
	\begin{equation*}
		\sum_{{k\ge0}}\sum_{\lambda\vdash k}\prod_{j=1}^k\frac{x_j^{m_j}}{m_j!}w^k=\exp\left(\sum_{{k\ge1}}x_k w^k\right).
	\end{equation*}
\end{lemma}

\subsection{The Dedekind eta function and (false) theta functions}\label{sub:false-and-partial}
First, we recall the multiplier for the \emph{Dedekind eta function} $\eta(\tau):=q^{\frac{1}{24}}\prod_{n\ge}(1-q^n)$, 
	\begin{align*}
\nu_\eta\! \begin{pmatrix}a&b\\c&d\end{pmatrix}:=\frac{\eta\Bigl(\frac{a\t+b}{c\t+d}\Bigr)}{ \sqrt{c\tau+d}\,\eta(\t)},\qquad\left(\begin{pmatrix} a&b \\ c&d \end{pmatrix}\in\mathrm{SL}_2(\Z)\right),
\end{align*}
which does not depend on the choice of $\tau$. 
Then we have the following transformation of~$\vartheta$. 
\begin{lemma}\label{lem:vartheta-eta-trans}
	For $\gamma=\begin{psmallmatrix} a&b \\ c&d \end{psmallmatrix}\in\mathrm{SL}_2(\Z)$ and $m,n\in\Z$, we have
	\begin{align*}
\vartheta \left(\frac{z}{c\tau+d};\frac{a\tau+b}{c\tau+d}\right) &= \nu_\eta^3 (\gamma) \sqrt{c\tau+d}   e^{\frac{\pi icz^2}{c\tau+d}} \vartheta(z;\tau),\\ 
\vartheta(z+m\tau+n;\tau) &= (-1)^{m+n}q^{-\frac{m^2}{2}}\zeta^{-m} \vartheta(z;\tau).
\end{align*}
\end{lemma}

The Jacobi theta function is an example of a Jacobi form.\footnote{For detailed study of Jacobi forms we refer the readers to Eichler--Zagier \cite{EZ1985}.} Another operator for functions transforming like Jacobi forms is the \emph{heat operator} defined, for $m\in\mathbb{Q}$, as
\begin{align}\label{eq:Heat-op}
	H_m:= 4m q\frac{\partial}{\partial q}-\left(\zeta \frac{\partial}{\partial \zeta}\right)^2.
\end{align}
This operator preserves the elliptic transformation of a Jacobi form of index $m$, and in special cases also the modularity, while increasing the weight by two.

Next, consider the false theta function from \cite{BN}
\begin{equation*}
	\psi(z;\tau) := i\sum_{n\in\Z} \sgn \left(n+\frac{1}{2}\right) (-1)^n \zeta^{n+\frac{1}{2}} q^{\frac{1}{2}\left(n+\frac{1}{2}\right)^2} 
\end{equation*}
Unlike the Jacobi theta function, $\psi$ is not modular due to the presence of an extra $\sgn$ factor. The first author and Nazaroglu found the following (Jacobi) completion of $\psi$ ($\tau,w\in\mathbb{H}$ and $z\in\mathbb{C}$)
\begin{equation*}
\widehat\psi(z;\tau,w) := i\sum_{n\in\Z} \erf \left(-i\sqrt{\pi i(w-\tau)} \left(n+\frac{1}{2}+\frac{z_2}{\tau_2}\right)\right) (-1)^n q^{\frac{1}{2}\left(n+\frac{1}{2}\right)^2} \zeta^{n+\frac{1}{2}},
\end{equation*}
where $\erf(z):=\frac{2}{\sqrt{\pi}}\int_0^ze^{-t^2}dt$. To recover $\psi$, one needs to take a certain limit.
More precisely, we have the following lemma \cite[Theorem~2.3~and~(1.3)]{BN}.
\begin{lemma}\label{lem:widehat-psi-trans}
	For $\gamma=\begin{psmallmatrix} a&b \\ c&d \end{psmallmatrix}\in\mathrm{SL}_2(\Z)$ and $m,n\in\Z$, we have 
	\begin{align*}
\widehat\psi\left(\frac{z}{c\tau+d};\frac{a\tau+b}{c\tau+d},\frac{aw+b}{cw+d}\right) &= \chi_{\tau,w} (\gamma) \nu_\eta^3 (\gamma) \sqrt{c\tau+d}   e^{\frac{\pi icz^2}{c\tau+d}} \widehat\psi(z;\tau,w),\\
	\widehat{\psi}(z+m\tau+n;\tau,w) &= (-1)^{m+n}q^{-\frac{m^2}{2}}\zeta^{-m} \widehat{\psi}(z;\tau,w).
	\end{align*}
Furthermore, $\psi$ is the holomorphic part of $\widehat\psi$ i.e., if $-\tfrac{1}{2}<\tfrac{z_2}{\tau_2}<\tfrac{1}{2}$ and $\varepsilon>0$ arbitrary, we have
\begin{equation}\label{limit}
\lim_{t\to\infty} \widehat\psi(z;\tau,\tau+it+\varepsilon) = \psi(z;\tau).
\end{equation}
\end{lemma}

\subsection{Quasimodular forms}\label{sub:quasimodular}
Recall Ramanujan's differential equations
\begin{align}\label{eq:der-of-Eis}
	D(G_2) &= - 2 G_2^2 + \frac56G_4,\quad
	D(G_4) = -8G_2G_4 + \frac7{10}G_6,\quad
	D\lrb{G_6} = -12G_2G_6 + \frac{400}{7}G_4^2.
\end{align}
We also need the well-known fact that 
\begin{align}\label{eq:der-eta-as-G-2}
	D\lrb{\Log(\eta)}=-G_2,
\end{align}
where $\Log$ denotes the principal branch of the complex logarithm. The function $G_2$ is an example of a quasimodular form.
A function $f$ on $\mathbb{H}$ is called a \emph{quasimodular form of weight $k$ and multiplier $\nu$ on} $\Gamma \subset \mathrm{SL}_2(\mathbb{Z})$ if there exist holomorphic functions $f_j : \mathbb{H} \to \mathbb{C}$ for $j \in \{0, \ldots, s\}$, for some $s \in \mathbb{N}_0$ with $f_0 = f$, such that the following hold: 
\begin{enumerate}[leftmargin=*]
\item We have
$
 f(\tfrac{a\tau + b}{c\tau + d})
= \nu \left(\begin{smallmatrix} a & b \\ c& d\end{smallmatrix}\right)(c\tau + d)^k\sum_{j=0}^s f_j(\tau)\bigl(\frac{c}{c\tau + d}\bigr)^j
\quad \text{for all }
\begin{psmallmatrix} a & b \\ c & d \end{psmallmatrix} \in \Gamma
$.\vspace{3pt}
\item As $\tau_2 \to \infty$, $\smash{(c\tau + d)^{-(k-2j)} f_j( \tfrac{a\tau + b}{c\tau + d})}$ is bounded for $j \in \{0, \ldots, s\}$ and $\begin{psmallmatrix} a & b \\ c & d \end{psmallmatrix} \in \mathrm{SL}_2(\mathbb{Z})$.
\end{enumerate}
A real-analytic function $f(\tau,w)$ \emph{transforms like a quasimodular form of weight $(k,\ell)$ and multiplier $\nu_{\tau,w}$ on $\Gamma\subset\mathrm{SL}_2(\Z)$ under the simultaneous action on $\tau$ and $w$} if there exist real-analytic functions $f_{n,m}(\tau,w)$ for $0\le m\le s_1,  0\le n\le s_2$ for some $s_1,s_2\in\N$ with $f_{0,0}(\tau,w)= f(\tau,w)$ such that, for all $\gamma=\begin{psmallmatrix}a&b\\c&d\end{psmallmatrix}\in\Gamma$,
\begin{align*}
f \left(\frac{a\tau+b}{c\tau+d},\frac{aw+b}{cw+d}\right) =\nu_{\tau,w} (\gamma)(c\tau+d)^{k}(cw+d)^{\ell}\sum_{\substack{0\le n\le s_1\\0\le m\le s_2}} f_{n,m}(\tau,w)\left(\frac{c}{c\tau+d}\right)^n\left(\frac{c}{cw+d}\right)^m.
\end{align*}

\section{Proofs of Theorems~\ref{thm:comp-of-g_k} and~\ref{thm:comp-of-h_k}}\label{sec:False and partial Eisenstein series}
In this section, we express $h$ in terms of $T$ and $\vartheta$,	 and give the completions of $T$ and $h$.
First, we note that
\begin{align}
T(z;\tau) &= \vartheta(z;\tau)+\psi(z;\tau), \quad
h(\z;q) = \frac{i}{2}\left(\zeta^{\frac12}-\zeta^{-\frac12}\right)q^{\frac{1}{8}}\sum_{\pm}\mp \zeta^{\pm1} T(3z\pm \tau;3\tau).\label{def:MathcalT-MathcalH-as-T}
\end{align}

\subsection{Modularity}
Using \eqref{def:MathcalT-MathcalH-as-T}, we define the completion of $T$ and $h$ as 
\begin{align}
\widehat T(z;\tau,w):=\widehat T(z,\overline{z};\tau,\overline{\tau},w,\overline{w})&:=\vartheta(z;\tau)+\widehat \psi (z;\tau,w),\label{That}\\
\widehat{H}(z;\tau,w):=\widehat{H}(z,\overline{z};\tau,\overline{\tau},w,\overline{w}) &:= \frac{i}{2}\left(\zeta^{\frac12}-\zeta^{-\frac12}\right)q^{\frac{1}{8}}\sum_{\pm}\mp \zeta^{\pm1} \widehat{T}(3z\pm \tau;3\tau,3w).\label{def:widehatMathcalH}
\end{align}

We have the following modular transformation of $\widehat{T}$.
\begin{proposition}\label{prop:Thattransfo}
	We have, for $\gamma=\begin{psmallmatrix} a&b\\ c&d \end{psmallmatrix}\in\mathrm{SL}_2(\Z)$, (with $\chi_{\tau,w}$ is defined in \eqref{def:chi-tau-w})
	\begin{equation*}
\widehat T \left(\frac{z}{c\tau+d};\frac{a\tau+b}{c\tau+d},\frac{aw+b}{cw+d}\right) = \chi_{\tau,w} (\gamma) \nu_\eta^3 (\gamma)\sqrt{c\tau+d}   e^{\frac{\pi icz^2}{c\tau+d}} \widehat T \left(\chi_{\tau,w} (\gamma) z;\tau,w\right).
	\end{equation*}
\end{proposition}
\begin{proof}
First we note, by \eqref{def:MathcalT-MathcalH-as-T}, using that $z\mapsto \vartheta(z;\tau)$ is odd and $z\mapsto\widehat{\psi}(z;\tau,w)$ is even,
\begin{align}\label{eq:T-1-T-2--z}
	\widehat T(-z;\tau,w) &= -\vartheta(z;\tau)+\widehat \psi (z;\tau,w).
\end{align}
Using \eqref{def:MathcalT-MathcalH-as-T} and Lemmas~\ref{lem:vartheta-eta-trans} and~\ref{lem:widehat-psi-trans}, we have 
\begin{equation*}
	\widehat T \left(\frac{z}{c\tau+d};\frac{a\tau+b}{c\tau+d},\frac{aw+b}{cw+d}\right) 
						= \nu_\eta^3(\gamma)\sqrt{c\tau+d}   e^{\frac{\pi icz^2}{c\tau+d}}\lrb{\vartheta(z;\tau)+ \chi_{\tau,w}(\gamma)\widehat \psi (z;\tau,w)}.
\end{equation*}
Recalling that $\chi_{\tau,w}$ takes value in $\{1,-1\}$ and using \eqref{That} and \eqref{eq:T-1-T-2--z} gives the proposition.
\end{proof}

Next we determine the modular transformation for $\widehat{H}$. We first define 
\begin{align}
	\widehat{\mathfrak{H}}(z;\tau,w)&:=\frac{2i\zeta^{\frac{1}{2}}q^{\frac{1}{24}}}{1-\zeta}\widehat{H}(z;\tau,w)\label{def:mathscrH}
\end{align}
\begin{proposition}\label{prop:WidehatH-transform}
	We have, for $\gamma=\begin{psmallmatrix} a&b\\ c&d \end{psmallmatrix}\in\Gamma_0(3)$ with $a\equiv\ell\Pmod{3}$ where $\ell\in\{1,-1\}$,
	\begin{align*}
	\widehat{\mathfrak{H}}\!\lrb{\frac{z}{c\t+d};\frac{a\t+b}{c\t+d}, \frac{aw+b}{cw+d}} &=  \nu (\gamma) \sqrt{c\tau+d}   e^{\frac{3\pi icz^2}{c\tau+d}} \widehat{\mathfrak{H}}\!\left(\chi_{\tau,w}(\gamma) z;\tau,w\right),
\end{align*}
where $\chi_{\tau,w}$ is defined in \eqref{def:chi-tau-w} and
\begin{align*}
	\nu(\gamma) := (-1)^{b+\frac{a-\ell}{3}} e^{\frac{\pi iab}{3}}\ell  \nu_\eta^3 \begin{pmatrix}a&3b\\\frac{c}{3}&d\end{pmatrix} .
\end{align*}
\end{proposition}
\begin{proof}
To determine the transformation of $\widehat{H}$, we first study $\widehat T(3z\pm \tau;3\tau,3w)$.  Using Proposition~\ref{prop:Thattransfo}, we have, for $\gamma=\begin{psmallmatrix} a&b\\ c&d \end{psmallmatrix}\in\Gamma_0(3)$,
\begin{multline}
	\widehat T\left(\frac{3z}{c\tau+d}\pm \frac{a\tau+b}{c\tau+d};3\frac{a\tau+b}{c\tau+d},3\frac{aw+b}{cw+d}\right)\\
	=\chi_{\tau,w}(\gamma)\, \nu_\eta^3\! \mat{a&3b\\\frac{c}{3}&d} \sqrt{c\tau+d}   e^{\frac{\pi ic \left(3z\pm (a\tau+b)\right)^2}{3(c\tau+d)}}  \widehat T \left(\chi_{\tau,w}(\gamma) (3z\pm (a\tau+b));3\tau,3w\right).
\label{Thattransform}
\end{multline}
Now, let
\begin{align*}
f_{\pm}(z;\tau,w):= q^{\frac{1}{6}}\z^{\pm1} \widehat{T}(3z\pm \tau;3\tau,3w).
\end{align*}
Writing $a=3r+\ell$   $(r\in \mathbb{Z}, \ell \in \{1,-1\})$ and using Lemmas~\ref{lem:widehat-psi-trans}, \ref{lem:vartheta-eta-trans}, and \eqref{def:MathcalT-MathcalH-as-T} gives
\begin{align*}
	\widehat T\left(3z\pm (a\tau+b);3\tau,3w\right)
	&=(-1)^{r+b} q^{-\frac{3r^2}{2}-\ell r} \zeta^{\mp3r} \widehat T(3z\pm \ell\tau;3\tau,3w).
\end{align*}
Using this and \eqref{Thattransform}, we have
\begin{align*}
	&f_{\pm} \left(\frac{z}{c\tau+d};\frac{a\tau+b}{c\tau+d},\frac{aw+b}{cw+d}\right) \nonumber\\
	&\hspace{1cm}=(-1)^{r+b} e^{\frac{\pi iab}{3}} e^{\frac{2\pi i \tau}{6}  \pm 2\pi i\ell z + \frac{3\pi ic z^2}{c\tau+d}} \chi_{\tau,w}(\gamma)\, \nu_\eta^3\! \mat{a&3b\\\frac{c}{3}&d} \sqrt{c\tau+d}\, \widehat T\!\left(\chi_{\tau,w}(\gamma) (3z\pm \ell\tau);3\tau,3w\right)\nonumber\\
&\hspace{1cm}= (-1)^{r+b} e^{\frac{\pi iab}{3}} \chi_{\tau,w}(\gamma)\, \nu_\eta^3\! \mat{a&3b\\ \frac{c}{3}&d} \sqrt{c\tau+d}\,    e^{\frac{3\pi i c z^2}{c\tau+d}} f_{\pm \ell \chi_{\tau,w}(\gamma)}\! \left(\chi_{\tau,w}(\gamma) z;\tau,w\right) . 
\end{align*}
The claim then follows by using \eqref{def:widehatMathcalH} and \eqref{def:mathscrH} to rewrite
\begin{align*}
	\widehat{\mathfrak{H}}(z;\tau,w) &= f_{-}(z;\t;w)- f_{+}(z;\t;w).\qedhere
\end{align*}
\end{proof}

\subsection{Modular completions of $g_k$ and $h_k$}
Analogous to \eqref{def:mathcalT-as-exp} and \eqref{def:mathcalH-as-exp}, we define
\begin{align}
	\widehat{T}(z;\tau,w) &=:  \widehat{T}(0;\tau,w)\exp\lrb{-\sum_{\substack{k\ge 1\\ \ell\ge0}}\widehat{g}_{k,\ell}(\tau,w)\frac{(2\pi iz)^k}{k!}\frac{(2\pi i\overline{z})^\ell}{\ell!}},\label{def:widehatMathcalT-as-exp}\\
	\widehat{H}(z;\tau,w) &=: 2i\sin(\pi z) \left[\frac{\partial}{\partial e^{2\pi ix}}\widehat{H}(x;\tau,w)\right]_{x=0} \exp\lrb{-\sum_{\substack{k\ge 1\\\ell\ge0}}\widehat{h}_{k,\ell}(\tau,w)\frac{(2\pi iz)^k}{k!}\frac{(2\pi i\overline{z})^\ell}{\ell!}}.\label{def:widehatMathcalH-as-exp}
\end{align}

Using Proposition~\ref{prop:Thattransfo} and \eqref{def:widehatMathcalT-as-exp}, we obtain the following theorem.

\begin{theorem}\label{thm:t-n-modular-trans}
	For $\gamma=\begin{psmallmatrix} a&b\\ c&d \end{psmallmatrix}\in\mathrm{SL}_2(\Z)$, we have
	\begin{align*}
			\widehat{g}_{k,\ell}\lrb{\frac{a\tau+b}{c\tau+d},\frac{aw+b}{cw+d}} &= 
			 \chi_{\tau,w}^{k+\ell} (\gamma) (c\tau+d)^k(c\overline{\tau}+d)^\ell 
			\widehat{g}_{k,\ell}(\tau,w) + \delta_{k,2}\delta_{\ell,0}\frac{ic}{2\pi}(c\tau+d).
	\end{align*}
\end{theorem}

Similarly, using \eqref{def:mathscrH}, Proposition~\ref{prop:WidehatH-transform} and \eqref{def:widehatMathcalH-as-exp} we get the following theorem.

\begin{theorem}\label{thm:h-k-ell-modular-trans}
	For $\gamma=\begin{psmallmatrix} a&b\\ c&d \end{psmallmatrix}\in\Gamma_0(3)$, we have
	\begin{align*}
			\widehat{h}_{k,\ell}\lrb{\frac{a\tau+b}{c\tau+d},\frac{aw+b}{cw+d}} &= 
			 \chi_{\tau,w}^{k+\ell} (\gamma) (c\tau+d)^k(c\overline{\tau}+d)^\ell 
			\widehat{h}_{k,\ell}(\tau,w)  + \delta_{k,2}\delta_{\ell,0}\frac{3ic}{2\pi}(c\tau+d) .
	\end{align*}
\end{theorem}

We next study the limiting behaviour of $\widehat{g}_{k,\ell}$ and $\widehat{h}_{k,\ell}$.
\begin{proposition}\label{limitbehaviour}
	For $\varepsilon>0$, we have
\begin{align*}
	\lim_{t\to\infty} \widehat{g}_{k,\ell}(\tau,\tau+it+\varepsilon) 
	&= \delta_{\ell,0} g_k(\tau),\qquad
	\lim_{t\to\infty} \widehat{h}_{k,\ell}(\tau,\tau+it+\varepsilon) 
	= \delta_{\ell,0} h_k(\tau).
\end{align*}
\end{proposition}
\begin{proof}
Using \eqref{def:MathcalT-MathcalH-as-T} and \eqref{limit}, we have
\begin{align}
	\lim_{t\to\infty}\widehat{T}(z;\tau,\tau+it+\varepsilon) &= T(z;\tau).\label{eq:mathcalT-T-limit-t-inf}
\end{align}
Similarly, we get, using \eqref{def:MathcalT-MathcalH-as-T}, \eqref{def:widehatMathcalH}, \eqref{eq:mathcalT-T-limit-t-inf}, and \eqref{eq:def-G-H},
\begin{align}
	\lim_{t\to\infty}\widehat{H}(z;\tau,\tau+it+\varepsilon) &= h(\z;q), \label{eq:mathcalH-H-limit-t-inf}
\qquad
	\lim_{t\to\infty}\left[\z\frac{\partial}{\partial \z}\widehat{H}(z;\tau,\tau+it+\varepsilon)\right]_{\z=1}  =-q^{-\frac{1}{24}}\eta(\tau).
\end{align}
Using \eqref{def:widehatMathcalT-as-exp}, \eqref{eq:mathcalT-T-limit-t-inf}, and \eqref{def:mathcalT-as-exp}, we obtain 
\begin{align}
	\lim_{t\to\infty} \widehat{T}(0;\tau,\tau+it+\varepsilon) & \exp\lrb{-\sum_{\substack{k\ge 1\\ \ell\ge0}}\widehat{g}_{k,\ell}(\tau,\tau+it+\varepsilon)\frac{(2\pi iz)^k}{k!}\frac{(2\pi i\overline{z})^\ell}{\ell!}} \nonumber\\
	&\hspace{5cm}= T_0(\tau)\exp\lrb{-\sum_{k\ge 1}g_k(\tau)\frac{(2\pi iz)^k}{k!}}.\label{eq:mathcalT-lim-mid-step}
\end{align}
Similarly, using \eqref{def:widehatMathcalH-as-exp}, \eqref{eq:mathcalH-H-limit-t-inf}, and \eqref{def:mathcalH-as-exp}, we get 
\begin{multline}
	   \lim_{t\to\infty} 2i\sin(\pi z) \left[\frac{\partial}{\partial e^{2\pi ix}}\widehat{H}(x;\tau,\tau+it+\varepsilon)\right]_{x=0}\exp\lrb{-\sum_{\substack{k\ge 1\\ \ell\ge0}}\widehat{h}_{k,\ell}(\tau,\tau+it+\varepsilon)\frac{(2\pi iz)^k}{k!}\frac{(2\pi i\overline{z})^\ell}{\ell!}}\\
	= -2i\sin(\pi z) q^{-\frac{1}{24}}\eta(\tau)\exp \lrb{-\sum_{k\ge 1}h_k(\tau)\frac{(2\pi iz)^k}{k!}} .  \label{eq:mathcalH-lim-mid-step}
\end{multline}
Using \eqref{eq:mathcalT-T-limit-t-inf} in \eqref{eq:mathcalT-lim-mid-step} and \eqref{eq:mathcalH-H-limit-t-inf} in \eqref{eq:mathcalH-lim-mid-step}, respectively, we obtain
\begin{align*}
	\sum_{\substack{k\ge 1\\\ell\ge0}}\lim_{t\to\infty}\widehat{g}_{k,\ell}(\tau,\tau+it+\varepsilon)\frac{(2\pi iz)^k}{k!}\frac{(2\pi i\overline{z})^\ell}{\ell!} &= \sum_{k\ge 1}g_k(\tau)\frac{(2\pi iz)^k}{k!},\\
	\sum_{\substack{k\ge 1\\\ell\ge0}}\lim_{t\to\infty}\widehat{h}_{k,\ell}(\tau,\tau+it+\varepsilon)\frac{(2\pi iz)^k}{k!}\frac{(2\pi i\overline{z})^\ell}{\ell!} &= \sum_{k\ge 1}h_k(\tau)\frac{(2\pi iz)^k}{k!}.
\end{align*}
Comparing coefficients gives the proposition.
\end{proof}

We are now ready to prove Theorems~\ref{thm:comp-of-g_k} and~\ref{thm:comp-of-h_k}.

\begin{proof}[Proofs of Theorems~\ref{thm:comp-of-g_k} and~\ref{thm:comp-of-h_k}]
	We define $\widehat{g}_k(\tau,w):=\widehat{g}_{k,0}(\tau,w)$ and $\widehat{h}_k(\tau,w):=\widehat{h}_{k,0}(\tau,w)$. 
	Then parts (1) of both theorems follow from Theorems~\ref{thm:t-n-modular-trans} and~\ref{thm:h-k-ell-modular-trans}. 
	Parts (2) of the theorems can be concluded from Proposition~\ref{limitbehaviour}.
\end{proof}

\section{Proofs of Theorems~\ref{thm:mathbbT-algebra} and~\ref{thm:h-n-diff}}\label{sec:differential-eq}
In this section, we show that $T$ and $h$ are eigenforms of the heat operators.
\subsection{Proof of Theorem~\ref{thm:mathbbT-algebra}}
A direct calculation shows that $T$ is annihilated by the heat operator~$H_{\frac{1}{2}}$ defined in \eqref{eq:Heat-op}.
\begin{lemma}\label{lem:PDE-for-mathcalT}
We have
	\begin{align*}
	H_{\frac12}\left(T(z;\tau)\right) &= 0.
	\end{align*}
\end{lemma}

\begin{proof}[Proof of Theorem~\ref{thm:mathbbT-algebra}] 
It suffices to prove the formulas for $D(g_k)$ and $D(\Log(T_0))$. 
From \eqref{def:mathcalT-as-exp},
\begin{align}
	&q\frac{\partial}{\partial q}T(z;\tau) = 
	 T(z;\tau)\lrb{q\frac{\partial}{\partial q}\Log\lrb{T_0(\tau)} -\sum_{k\ge 1} q\frac{\partial}{\partial q}g_k(\tau)\frac{(2\pi iz)^k}{k!}}. \label{eq:H2-q-action}
\end{align}
Furthermore, we have
\begin{align}
	&\lrb{\z\frac{\partial}{\partial \z}}^2 T(z;\tau)
	= T(z;\tau)\left( -\sum_{k\ge 1} k(k-1)g_k(\tau)\frac{(2\pi iz)^{k-2}}{k!}  +  \left(\sum_{k\ge 1}kg_k(\tau)\frac{(2\pi iz)^{k-1}}{k!}\right)^{  2}  \right) .\label{eq:H2-zeta-action}
\end{align}
Using Lemma~\ref{lem:PDE-for-mathcalT}, \eqref{eq:H2-q-action}, and \eqref{eq:H2-zeta-action} gives 
\begin{align*}
	H_{\frac{1}{2}}(T(z;\tau))&=T(z;\tau)
	\lrb{2\lrb{q\frac{\partial}{\partial q}\Log\lrb{T_0(\tau)} -\sum_{k\ge 1} q\frac{\partial}{\partial q}g_k(\tau)\frac{(2\pi iz)^k}{k!}}\right.\\
		&\qquad\left.+\sum_{k\ge 1} k(k-1)g_k(\tau)\frac{(2\pi iz)^{k-2}}{k!}  - \left(\sum_{k\ge 1}kg_k(\tau)\frac{(2\pi iz)^{k-1}}{k!}\right)^{  2}  }=0.
\end{align*}
Since $T(z;\tau)\neq 0$, this is equivalent to
\begin{align*}
	2\sum_{k\ge 1} q\frac{\partial}{\partial q}g_k(\tau)\frac{(2\pi iz)^k}{k!} &= \sum_{k\ge 0}g_{k+2}(\tau)\frac{(2\pi iz)^{k}}{k!} 
	 - \lrb{\sum_{k\ge 0}g_{k+1}(\tau)\frac{(2\pi iz)^{k}}{k!}}^{ 2} + 2q\frac{\partial}{\partial q}\Log\lrb{T_0(\tau)}.
\end{align*}
Comparing the coefficients of $(2\pi iz)^k$ gives the theorem.
\end{proof}

\subsection{Proof of Theorem~\ref{thm:h-n-diff}}
We show that a multiple of $h$ is annihilated by the heat operator. More precisely, define
\begin{align}\label{def:mathfrackH}
	\mathfrak{H}(z;\tau):=\frac{2iq^{\frac{1}{24}}\zeta^{\frac{1}{2}}h(\z;q)}{1-\zeta}.
\end{align} 
\begin{lemma}\label{lem:H-3/2-on-mathcalH}
We have 
\[
	H_{\frac{3}{2}} (\mathfrak{H}(z;\tau) ) = 0.
\]
\end{lemma}
\begin{proof}
From \eqref{eq:def-G-H}, we obtain 
	\begin{align}
		\mathfrak{H}(z;\tau) = 2i\sum_{n\ge0} (-1)^n \zeta^{3n+\frac{1}{2}} q^{\frac{1}{6}\left(3n+\frac{1}{2}\right)^2} - 2i\sum_{n\ge0} (-1)^n \zeta^{3n+\frac{5}{2}} q^{\frac{1}{6}\left(3n+\frac{5}{2}\right)^2}. \label{eq:mathcalH-as-two-sums}
	\end{align}
	We have that
	\begin{align*}
		H_{\frac{3}{2}}\!\left(\zeta^{3n+\frac{1}{2}} q^{\frac{1}{6}\left(3n+\frac{1}{2}\right)^2}\right) = 
		H_{\frac{3}{2}}\!\left(\zeta^{3n+\frac{5}{2}} q^{\frac{1}{6}\left(3n+\frac{5}{2}\right)^2} \right) =0.
	\end{align*}
	Plugging these into \eqref{eq:mathcalH-as-two-sums} gives the claim.
\end{proof}

\begin{proof}[Proof of Theorem~\ref{thm:h-n-diff}] 
By \eqref{def:mathfrackH} and \eqref{def:mathcalH-as-exp}, we have 
\begin{align}
	\mathfrak{H}(z;\tau)  = 2i\eta(\tau)\exp\lrb{-\sum_{k\ge 1}h_k(\tau)\frac{(2\pi iz)^k}{k!}}.\label{eq:H32-on-mathcalH-mid-step}
\end{align} 
Using this, we compute  
\begin{align}
	&q\frac{\partial}{\partial q} \mathfrak{H}(z;\tau)= \mathfrak{H}(z;\tau)  \lrb{q\frac{\partial}{\partial q} \Log\lrb{\eta(\tau)}-\sum_{k\ge 1}q\frac{\partial}{	\partial q}h_k(\tau)\frac{(2\pi iz)^k}{k!}}.\label{eq:H32-q-action}
\end{align}
Next, we have 
\begin{align}
\hspace{-3pt}\lrb{\z\frac{\partial}{\partial\z}}^{ 2}  \mathfrak{H}(z;\tau) = \mathfrak{H}(z;\tau) 
	\lrb{ \lrb{\sum_{k\ge 0}h_{k+1}(\tau)\frac{(2\pi iz)^{k}}{k!}}^{  2} -\sum_{k\ge 0}h_{k+2}(\tau)\frac{(2\pi iz)^{k}}{k!}   } .\label{eq:H32-zeta-action}
\end{align}
Plugging \eqref{eq:H32-q-action}, \eqref{eq:H32-zeta-action}, and \eqref{eq:H32-on-mathcalH-mid-step} into Lemma~\ref{lem:H-3/2-on-mathcalH} then gives 
\begin{align*}
	H_{\frac{3}{2}}\left(\mathfrak{H}(z;\tau) \right)&=  \lrb{6q\frac{\partial}{\partial q} \Log\lrb{\eta(\tau)}-6\sum_{k\ge 1}q\frac{\partial}{	\partial q}h_k(\tau)\frac{(2\pi iz)^k}{k!} +\sum_{k\ge 0}h_{k+2}(\tau)\frac{(2\pi iz)^{k}}{k!} \right.\\
	&\hspace{5.5cm}\left.- \lrb{\sum_{k\ge 0}h_{k+1}(\tau)\frac{(2\pi iz)^{k}}{k!}}^{ 2}}\mathfrak{H}(z;\tau)=0.
\end{align*}
Since $\mathfrak H(z;\tau)\neq 0$,
by comparing the coefficient of $(2\pi iz)^k$ for $k\in\N$, we get 
\begin{align}
	6q\frac{\partial}{\partial q}h_k(\tau)&= h_{k+2}(\tau) -\sum_{d=0}^k \binom{k}{d} h_{d+1}(\tau)h_{k-d+1}(\tau).\label{eq:D-h-1-recur}
\end{align}
Furthermore, comparing the constant term and using \eqref{eq:der-eta-as-G-2} gives
\begin{align}\label{eq:h-2-h-1-G-2}
	-G_2(\tau) = q\frac{\partial}{\partial q} \Log\lrb{\eta(\tau)} &= -\frac{h_2(\tau)}{6} + \frac{h_1^2(\tau)}{6}.
\end{align}

It remains to show that $\mathcal{H} =\widetilde{\mathcal M}[h_1,h_3,h_5,\ldots]$. Using \eqref{eq:h-2-h-1-G-2}, \eqref{eq:D-h-1-recur}, and \eqref{eq:der-of-Eis}, we have $G_2,G_4,G_6\in\mathcal{H}=\C[h_1,h_2,\ldots]$.
To finish the proof, we need to show that $h_k$ for $k$ even can be written as a polynomial in $G_2,G_4,G_6$, and $h_j$ with $j$ odd. This follows inductively by differentiating \eqref{eq:h-2-h-1-G-2} and using \eqref{eq:D-h-1-recur} and the fact that the space of quasimodular forms is closed under differentiating. 
\end{proof}

\section{Proofs of Theorems~\ref{thm:Fourier-exp-of-g_k} and~\ref{thm:Fourier-exp-of-h_k}}\label{sec:Fourier-exp}
In this section, we give explicit formulas for the Fourier coefficients of $g_k$ and $h_k$.

\subsection{Constant terms} We first compute the constant terms of $g_k$ and $h_k$.
\begin{lemma}\label{lem:thlimit}
	We have
	\begin{align*}
	\lim_{\tau\to i\infty} g_k(\tau) = \lim_{\tau\to i\infty}h_k(\tau) = -\frac{\delta_{k,2}}{2}.
	\end{align*}
\end{lemma}
\begin{proof}
From \eqref{def:mathcalT-as-exp} and \eqref{def:mathcalH-as-exp}, we have
\begin{equation*}
	\frac{T(z;\tau)}{T_0(\tau)} = \exp\lrb{-\sum_{k\ge 1}g_k(\tau)\frac{(2\pi iz)^k}{k!}},\qquad
	\frac{iq^{\frac{1}{24}}h(\z;q)}{2\sin(\pi z)\eta(\tau)} =  \exp\lrb{-\sum_{k\ge 1}h_k(\tau)\frac{(2\pi iz)^k}{k!}}.
\end{equation*}
Using \eqref{def:mathcalT} and \eqref{eq:def-G-H}, note that 
\begin{align*}
	T(z;\tau) = 2i\z^{\frac12}q^{\frac18}(1+O_\z(q)),\qquad h(\z;q)=1-\z + O_\z(q),\qquad q^{-\frac{1}{24}}\eta(\tau)=1+O(q).
\end{align*}
Hence we have
\begin{align*}
	\exp\lrb{-\sum_{k\ge 1} \lim_{\tau\to i\infty}g_k(\tau)\frac{(2\pi iz)^k}{k!}}&= \lim_{\tau\to i\infty} \frac{T(z;\tau)}{T_0(\tau)}=\z^{\frac12}=\exp\lrb{\frac12\frac{2\pi i z}{1!}},\\
	\exp\lrb{-\sum_{k\ge 1} \lim_{\tau\to i\infty}h_k(\tau)\frac{(2\pi iz)^k}{k!}} &= \lim_{\tau\to i\infty}\frac{ih(\z;q)}{2\sin(\pi z)q^{-\frac{1}{24}}\eta(\tau)} = \frac{\z-1}{2 i \sin(\pi z)}=
	\z^{\frac{1}{2}}=\exp\lrb{\frac{1}{2} \frac{2\pi i z}{1!}}.
\end{align*}
Taking the logarithm of these and comparing coefficients gives the lemma.
\end{proof}

\subsection{Proof of Theorem~\ref{thm:Fourier-exp-of-g_k}}\hspace{0cm} 
\begin{proof}[Proof of Theorem~\ref{thm:Fourier-exp-of-g_k}]
From \eqref{def:mathcalT-as-exp}, we have, 
\begin{align}
	\sum_{k\ge 0}g_k(\tau)\frac{(2\pi iz)^k}{k!}=-\Log\lrb{\mathtt{T}(\z;q)}\label{eq:Log-of-MathcalT}
\end{align}
with $g_0(\tau):=-\Log\lrb{\mathtt{T}(1;q)}$ and 
\begin{align}\label{T def}
	\mathtt{T}(\z;q):= \frac{T(z;\tau)}{2iq^{\frac18}}= \sum_{n\ge0} (-1)^n \zeta^{n+\frac{1}{2}} q^{\frac{n(n+1)}{2}}.
\end{align}

Applying \eqref{eq:Fa-di-bruno} to $f(q)=\Log(q)$ and  $g(q)=\mathtt{T}(\zeta;q)$, we obtain, for $n\in \N$, that
\begin{align*}
	\hspace{-0.15cm}{\partial^{n} \over \partial q^{n}}\Log (\mathtt{T}(\zeta;q))
	&=\sum_{\lambda \vdash n} \frac{n!\big(\sum_{r=1}^n m_r -1\big)!(-1)^{\sum_{r=1}^n m_r +1}}{m_{1}!m_{2}!\cdots m_{n}!\mathtt{T}(\z;q)^{\sum_{r=1}^n m_r}}  \prod _{j=1}^{n}\left(\frac{\frac{\partial^j}{\partial q^j}\mathtt{T}(\zeta;q)}{j!}\right)^{m_{j}}   .
\end{align*}
Note that from \eqref{T def} $\mathtt{T}(\zeta;0)=\zeta^\frac12$. Dividing by $n!$ and inserting the definition of $\mathtt{T}$, we obtain 
\begin{multline*}
	\operatorname{coeff}_{[q^n]}\Log(\mathtt{T}(\zeta;q)) \\
	=\lim\limits_{q\to 0^+}\sum_{\lambda \vdash n} {\frac {(-1)^{\sum_{r=1}^n m_r +1}}{\sum_{j=1}^n m_j}}\zeta^{-\frac12\sum_{r=1}^nm_r}\binom{\sum_{r=1}^n m_r }{m_1,m_2,\ldots,m_n} \prod _{j=1}^{n}\left(\frac{\frac{\partial^j}{\partial q^j}\sum_{n\ge0} (-1)^n \zeta^{n+\frac{1}{2}} q^{\frac{n(n+1)}{2}}}{j!}\right)^{m_{j}}.
\end{multline*}
Let $T_\ell:=\frac{\ell(\ell+1)}{2}$ be the \emph{$\ell$-th triangular number} and denote by $\mathscr{T}(n)$ the set of all partitions of $n$ into triangular numbers. For a partition $\lambda \in \mathscr{T}(n)$, we write $\lambda = ( T_1^{m_{T_1}}, T_2^{m_{T_2}},\ldots)$. Noting that $\sum_{r=1}^n m_r = \sum_{\ell\ge1} m_{T_\ell}$, we have 
\begin{align}
	\operatorname{coeff}_{[q^n]}\Log(\mathtt{T}(\zeta;q))& =\sum_{\lambda \in \mathscr{T}(n)} \frac {(-1)^{\sum_{\ell\ge1} m_{T_\ell} +1}\zeta^{-\frac12\sum_{\ell\ge1} m_{T_\ell}}}{\sum_{\ell\ge1} m_{T_\ell}} \binom{\sum_{\ell\ge1} m_{T_\ell} }{m_{T_1},m_{T_2},\ldots}
	 \prod_{\ell\ge1}\left((-1)^\ell \zeta^{\ell+\frac{1}{2}}\right)^{m_{T_\ell}} \nonumber\\
	&=\sum_{\lambda \in \mathscr{T}(n)} \frac {(-1)^{\sum_{\ell\ge1} m_{T_\ell} +1}}{\sum_{\ell\ge1} m_{T_\ell}}\binom{\sum_{\ell\ge1} m_{T_\ell} }{m_{T_1},m_{T_2},\ldots} \left(-\zeta\right)^{\sum_{\ell\ge1}\ell m_{T_\ell}} .\label{coeff}
\end{align}
Let $\mathscr{G}(n)$ be the set of partitions with non-increasing multiplicities i.e., $\lambda=(1^{m_1},2^{m_2},\ldots,n^{m_n})\vdash n$ with $m_1\geq m_2\geq \ldots \geq m_n \geq 0$. We define a bijection $\mathscr{T}(n)\to \mathscr{G}(n)$
\begin{align*}
	\left( T_1^{m_{T_1}}, T_2^{m_{T_2}},T_3^{m_{T_3}},\ldots \right)\mapsto \left( 1^{\ell(\lambda)},2^{\ell(\lambda)-m_{T_1}},3^{\ell(\lambda)-m_{T_1}-m_{T_2}},\ldots \right)
\end{align*}
by writing each part $T_\ell$ of $\lambda$ as  $T_\ell=\sum_{j=1}^{\ell} j$.
The inverse map is given by 
\begin{align*}
	&\left( 1^{m_1},2^{m_2},\ldots, n^{m_n} \right) \mapsto \left( T_1^{m_1-m_2}, T_2^{m_2-m_3},\ldots\right).
\end{align*}
 Hence, \eqref{coeff} simplifies as 
\begin{align}\label{eq:Tqcoef}
	\operatorname{coeff}_{[q^n]}\Log(\mathtt{T}(\zeta;q)) =\sum_{\lambda\in \mathscr{G}(n)} {\frac {(-1)^{m_1 +1}}{m_{1}}}\binom{m_1 }{m_1-m_2,m_2-m_3,\ldots} (-\zeta)^{\sum_{j\ge1}m_j} .
\end{align}
Using \eqref{eq:Log-of-MathcalT}, Lemma~\ref{lem:thlimit}, and \eqref{eq:Tqcoef}, we have 
\begin{align*}
g_k(\tau) 
&= -\frac{\delta_{k,1}}{2}-{k!}\operatorname{coeff}_{\left[(2\pi iz)^k\right]}\sum_{n\geq 1}\sum_{\lambda\in \mathscr{G}(n)} {\frac {(-1)^{m_1 +1}}{m_{1}}}\binom{m_1 }{m_1-m_2,m_2-m_3,\ldots} \left(-\zeta\right)^{\sum_{j\ge1} m_j} q^n\\
&=-\frac{\delta_{k,1}}{2}+\sum_{n\geq 1} \sum_{\lambda\in \mathscr{G}(n)} {\frac {(-1)^{\sum_{j\ge2} m_j}}{m_{1}}}\binom{m_1 }{m_1-m_2,m_2-m_3,\ldots} \sum_{j\ge1}m_j \left(\sum_{j\ge1} m_j\right)^{k-1} q^{n}.
\end{align*}
Note that, using Lemma~\ref{lem:div}, we have
\[
{\frac {(-1)^{\sum_{j\ge2}m_j}}{m_{1}}}  \sum_{j\ge1}m_j  \binom{m_1 }{m_1-m_2,m_2-m_3,\ldots}\in\Z. 
\]
Hence, we have, with $a_{n,m}\in \mathbb{Z}$ as in the theorem, 
\begin{align*}
g_k(\tau) = {-\frac{\delta_{k,1}}{2}}+\sum_{n\ge1} \sum_{m=1}^n a_{n,m} {m^{k-1}} q^n.
\end{align*} 

We claim that the conditions in $\Lambda(n,m)$ give
\begin{equation*}
m=\sum_{j=1}^n m_j \ge \frac{1}{2}\left(\sqrt{8n+1}-1\right).
\end{equation*}
Indeed, let $\ell\in\N$ be the unique integer such that $T_\ell\leq n<T_{\ell+1}$. 
Recall the constraints
$
m_1 \ge m_2 \ge \cdots \ge m_n \ge 0$, and $\sum_{j=1}^n j m_j = n. 
$
Hence, if $m_{\ell+1}\geq 1$, we find $m_1, \ldots,m_{\ell+1}\geq 1$ and therefore $n=\sum_{j=1}^n j m_j \geq \sum_{j=1}^{\ell+1} j = T_{\ell+1}>n$, a contradiction. We conclude that $m_{\ell+1}=0$. Furthermore, we also have that $0\le m_\ell\le 1$, since otherwise $m_1\ge m_2\ge \cdots\ge m_\ell\ge2$ which gives $n=\sum_{j=1}^\ell jm_j\ge 2\sum_{j=1}^\ell j=2T_\ell$, again a contradiction. 
Using Lemma~\ref{lem:ineq} with $n_j=m_j$, $a_j=j,$ and $b_j=\frac{\ell}{2}$ for $j\in\{1,\ldots,\ell-1\}$ as well as $m_\ell\leq 1$, we find that
\begin{align}
n& 
\le \frac{\ell}{2}m +\frac{\ell}{2}, \qquad \text{or equivalently} \qquad
\label{eq:m+1}
	m+1\geq \frac{2n}{\ell}
\end{align}
As $T_\ell\leq n$ (and $\ell>0$), we find $\ell \leq \frac{1}{2}(\sqrt{8n+1}-1)$.
Plugging this into \eqref{eq:m+1} gives $m\ge \frac{1}{2} (\sqrt{8 n+1}-1)$.
Thus $a_{n,m}=0$ unless $\frac{1}{2}(\sqrt{8n+1}-1) \le m \le n$.
This proves the theorem. 
\end{proof}

\subsection{Proof of Theorem~\ref{thm:Fourier-exp-of-h_k}}
\begin{proof}[Proof of Theorem~\ref{thm:Fourier-exp-of-h_k}] 
Using \eqref{def:mathcalH-as-exp}, we have
\begin{align}
	\sum_{k\ge 0}h_k(\tau)\frac{(2\pi iz)^k}{k!} &= -\Log\lrb{\mathtt{H}(\z;q)}\label{eq:mathcal-H-mid-2}
\end{align}
with $h_0(\tau):=-\Log(q^{-\frac{1}{24}}\eta(\tau))$, and 
\begin{align}\label{def:mathttH}
	\mathtt{H}(\z;q):= \frac{ih(\z;q)}{2\sin(\pi z)}.
\end{align}
Applying \eqref{eq:Fa-di-bruno} with $f(q)=\Log(q)$ and $g(q)=\mathtt{H}(\z;q)$, we obtain, for $n\in\N_0$,
\begin{equation}
	{\partial^{n} \over \partial q^{n}}\Log\lrb{\mathtt{H}(\z;q)} =\sum_{\lambda \vdash n} \frac {(-1)^{\sum_{r=1}^n m_r +1}n! \big(\sum_{r=1}^n m_r -1\big)!}{m_{1}!  m_{2}!  \cdots   m_{n}! \mathtt{H}(\z;q)^{\sum_{r=1}^n m_r}}
	 \prod _{j=1}^{n}\left(\frac{\frac{\partial^j}{\partial q^j}\mathtt{H}(\z;q)}{j!}\right)^{m_{j}} .\label{eq:n-th-der-wrt-q-log-mathcalH}
\end{equation}
Using \eqref{eq:def-G-H} and \eqref{def:mathttH}, we get
\begin{align*}
	\mathtt{H}(\z;q)&= \zeta^{\frac12} \sum_{n\ge0} (-1)^n \zeta^{3n} q^{\frac{n(3n+1)}{2}} + \zeta^{-\frac12}\sum_{n\ge1} (-1)^n \zeta^{3n} q^{\frac{n(3n-1)}{2}}.
\end{align*}
From this, we obtain that
\begin{align*}
	\left[\frac{\partial^j}{\partial q^j}\mathtt{H}(\z;q)\right]_{q=0} &= \begin{cases}
								\zeta^{\frac12} & \text{if }j=0,\\[.4em]
								(-1)^{\ell}\z^{3\ell-\frac12} P_{\ell}! &\text{if } j=P_{\ell} \text{ with }\ell\in\N,\\[.4em]
								(-1)^{\ell}\z^{3\ell+\frac{1}{2}} P_{-\ell}! & \text{if } j=P_{-\ell} \text{ with } \ell\in\N,\\[.4em]
								0 &\text{otherwise},
							\end{cases}
\end{align*}
where, for $\ell\in \Z\setminus\{0\}$, we let $P_{\ell}:=\tfrac{\ell(3\ell- 1)}{2}$ be the \emph{$\ell$--th generalized pentagonal number}. Moreover, denote by $\mathscr{P}(n)$ the set of partitions of $n$ into generalized pentagonal numbers. We write a partition $\lambda\in \mathscr{P}(n)$ as $\lambda=(P_{1}^{m_{P_{1}}},P_{-1}^{m_{P_{-1}}},P_{2}^{m_{P_{2}}},P_{-2}^{m_{P_{-2}}},\ldots)$. Note that we have $\sum_{r=1}^n m_r=\sum_{\ell\ge1} (m_{P_\ell}+m_{P_{-\ell}})$. With this notation, \eqref{eq:n-th-der-wrt-q-log-mathcalH} becomes 
\begin{align}
	\operatorname{coeff}_{[q^n]}\Log\lrb{\mathtt{H}(\z;q)} &=\sum_{\lambda \in \mathscr{P}(n)} \frac {(-1)^{\sum_{\ell\ge1} \left(m_{P_\ell}+m_{P_{-\ell}}\right) +1}}{\sum_{\ell\ge1} \left(m_{P_\ell}+m_{P_{-\ell}}\right)}\binom{\sum_{\ell\ge1} \left(m_{P_\ell}+m_{P_{-\ell}}\right) }{m_{P_1},m_{P_{-1}},m_{P_2},m_{P_{-2}},\ldots}\label{eq:h-n-q-coeff}\\
	&\qquad\times\z^{-\frac{1}{2}\sum_{\ell\ge1} \left(m_{P_\ell}+m_{P_{-\ell}}\right)} \prod _{\ell\ge1}\left((-1)^{\ell}\zeta^{3\ell-\frac12}\right)^{m_{P_{\ell}}}
	 \prod _{\ell\ge1} \left((-1)^{\ell} \zeta^{3\ell+\frac12}\right)^{m_{P_{-\ell}}} \nonumber\\
	 &\hspace{-3cm}=\sum_{\lambda \in \mathscr{P}(n)} \frac {(-1)^{\sum_{\ell\ge1}(\ell+1) \lrb{m_{P_{\ell}}+m_{P_{-\ell}}}+1}}{\sum_{\ell\ge1} (m_{P_\ell}+m_{P_{-\ell}})}\binom{\sum_{\ell\ge1} \left(m_{P_\ell}+m_{P_{-\ell}}\right) }{m_{P_1},m_{P_{-1}},m_{P_2},m_{P_{-2}},\ldots} \z^{3\sum_{\ell\ge1} \ell\lrb{m_{P_{\ell}}+m_{P_{-\ell}}} - \sum_{\ell\ge1} m_{P_{\ell}}} .\nonumber
\end{align}

Next, denote by $\mathscr{H}(n)$ the set of partitions of $n$ having no parts that are multiple of three and such that $m_{3j-2}\ge m_{3j+1}$ and $m_{3j-1}\ge m_{3j+2}$ for $j\geq 1$. Define a bijection $\mathscr{P}(n)\to \mathscr{H}(n)$ by 
\[
	\left(P_{1}^{m_{P_{1}}},P_{-1}^{m_{P_{-1}}},P_{2}^{m_{P_{2}}},P_{-2}^{m_{P_{-2}}},\ldots\right)
	\mapsto \left( 1^{\sum_{\ell\ge1} m_{P_\ell}}, 2^{\sum_{\ell\ge1} m_{P_{-\ell}}}, 4^{\sum_{\ell\geq 2} m_{P_\ell}}, 5^{\sum_{\ell\geq 2} m_{P_{-\ell}}},\ldots \right),
\]
where we write for each part $P_{\ell}=\sum_{m=1}^\ell (3m-2)$ and $P_{-\ell}=\sum_{m=1}^\ell (3m-1)$ for $1\le \ell\le n$. The inverse map is given by
\begin{align*}
	\left( 1^{m_1},2^{m_2},4^{m_4},5^{m_5},\ldots \right) \mapsto \left( P_1^{m_1-m_4}, P_{-1}^{m_2-m_5}, P_{2}^{m_4-m_7}, P_{-2}^{m_5-m_8},\ldots \right).
\end{align*}
Therefore \eqref{eq:h-n-q-coeff} becomes
\begin{align}
	&\operatorname{coeff}_{[q^n]}\Log\lrb{\mathtt{H}(\z;q)} 
	= \sum_{\lambda\in \mathscr{H}(n)} \frac{(-1)^{\sum_{\ell\ge1} (m_{3\ell-1}+m_{3\ell-2}) + m_1+m_2+1}}{m_1+m_2}
	\binom{m_1+m_2}{m_1-m_4,m_2-m_5,\ldots}\nonumber\\&\hspace{10cm}\times\zeta^{3\sum_{\ell\ge1} (m_{3\ell-1}+m_{3\ell-2})-m_1} .\label{eq:log-mathcalH-mid}
\end{align}
By Lemma~\ref{lem:thlimit}, \eqref{eq:mathcal-H-mid-2}, and \eqref{eq:log-mathcalH-mid}, we have, for $k\in\N$,
\begin{align*}
h_k(\tau)
	&= -\frac{\delta_{k,1}}{2}+ \sum_{n\geq 1} \sum_{\lambda\in \mathscr{H}(n)} \!(-1)^{\sum_{\ell\ge2} (m_{3\ell-1}+m_{3\ell-2})}\frac{3\sum_{\ell\ge1} (m_{3\ell-1}+m_{3\ell-2})-m_1}{m_1\!+\!m_2}\binom{m_1+m_2}{m_1-m_4,m_2-m_5,\ldots} \\
	& \hspace{9cm}\times\lrb{3\sum_{\ell\ge1} (m_{3\ell-1}+m_{3\ell-2})-m_1}^{\!k-1}\! q^n.
\end{align*}
Now let $d:=\gcd(m_1-m_4,m_2-m_5,\ldots)$.  By \cite[Lemma~2.8]{BPvI25}, we have $\frac{m_1+m_2}{d}|\binom{m_1+m_2}{m_1-m_4,m_2-m_5,\ldots}$.
Moreover, note that $d$ divides
\begin{multline*}
\hspace{-0.3cm}3\lrb{m_1-m_4+2(m_4-m_7)+3(m_7-m_{10})+\ldots}+3\lrb{m_2-m_5+2(m_5-m_8)+3(m_8-m_{11})+\ldots}\\- \lrb{m_1-m_4+m_4-m_7+\ldots} = 3\sum_{\ell\ge1} (m_{3\ell-1}+m_{3\ell-2})-m_1.
\end{multline*}
Thus, $m_1+m_2=\frac{m_1+m_2}{d}d$ divides $(3\sum_{\ell\ge1} (m_{3\ell-1}+m_{3\ell-2})-m_1)\binom{m_1+m_2}{m_1-m_4,m_2-m_5,\ldots}$. Hence we get
\begin{align*}
	h_k(\tau) &= -\frac{\delta_{k,1}}{2} + \sum_{n,m\geq 1} b_{n,m} m^{k-1} q^n,
\end{align*}
where $b_{n,m}\in \Z$ is defined as in the theorem.

To finish the proof, we have to show that 
\begin{align}
	\left\lfloor \frac{1}{2} \left(\sqrt{24 n+1}-1\right) \right\rfloor \le m\le 2n. \label{eq:bound-on-m}
\end{align}
The definition of $\Omega(n,m)$ directly gives that $m\le 2n$.
Next we show the lower bound for $m$. It is not hard to see that for fixed $n$, there exists a unique $\ell\in\N_0$ such that either $P_{-\ell}\le n<P_{\ell+1}$ (which we refer to as (1)) or $P_{\ell+1}\le n<P_{-\ell-1}$ (which we refer to as (2)).
Recall that (by the condition in $\Omega(n,m)$) $m_1\geq m_4\geq m_7 \geq \ldots\geq 0$, $m_2\geq m_5\geq m_8 \geq \ldots\geq 0$ and $\sum_{j=1}^n jm_j=n$. If $m_{3\ell+4}\ge 1$, then 
\[
n=\sum_{j=1}^n jm_j\ge \sum_{j=1}^{\ell+2} (3j-2)m_{3j-2}\ge \sum_{j=1}^{\ell+2} (3j-2) =P_{\ell+2}>P_{-\ell-1}>n,
\]
a contradiction. Similarly, if $m_{3\ell+2}\ge 1$, then 
\[
n=\sum_{j=1}^n jm_j\ge \sum_{j=1}^{\ell+1} (3j-1)m_{3j-1}\ge \sum_{j=1}^{\ell+1} (3j-1)=P_{-\ell-1}>n, 
\] 
again a contradiction. Finally, if $m_{3\ell+1}\ge 1$ then 
\[
n=\sum_{j=1}^n jm_j\ge \sum_{j=1}^{\ell+1} (3j-2)m_{3j-2}\ge \sum_{j=1}^{\ell} (3j-2)=P_{\ell+1},
\] 
which is a contradiction if we are in case~(1). 
This gives that $m_j=0$ for all $j\ge 3\ell+2$ in case (2) and for all $j\geq 3\ell+1$ in case (1). 
Applying Lemma~\ref{lem:ineq} with $n_j=m_{3j-1}, a_j=3j-1$, and $b_j=\frac{3\ell+1}{2}$, we obtain
\begin{align}\label{eq:2mod3}
\sum_{j=1}^{\ell} (3j-1)m_{3j-1} \leq  \frac{3\ell+1}{2}\sum_{j=1}^\ell m_{3j-1}.
\end{align}
Similarly applying Lemma~\ref{lem:ineq} with $r=\ell+1$, $n_j=m_{3j-2}$, $a_j=3j-2$, and $b_j=\frac{\ell}{2}(3-\delta_{j,1})$  (resp.\ $b_j = \frac{\ell+1}{2}(3-\delta_{j,1})$) if we are in case (1) (resp.\ (2)), we get
\begin{align}\label{eq:1mod3}
\sum_{j=1}^{\ell+1} (3j-2)m_{3j-2} \leq \begin{cases} \frac{3\ell}{2}\left(\frac{2}{3}m_1+\sum_{j=2}^{\ell+1} m_{3j-2}\right) & \hspace{-0.2cm}\text{if we are in case (1)}, \\[10pt]
\frac{3\ell+3}{2}\left(\frac{2}{3}m_1+\sum_{j=2}^{\ell+1} m_{3j-2}\right) & \hspace{-0.2cm}\text{if we are in case (2)}.
\end{cases}
\end{align}
By combining \eqref{eq:2mod3} and \eqref{eq:1mod3}, we deduce
\begin{align}
n&=\sum_{j=1}^{3\ell+1} j m_j = \sum_{j=1}^{\ell} (3j-1)m_{3j-1} + \sum_{j=1}^{\ell+1} (3j-2)m_{3j-2} \leq\begin{cases}
\frac{(3\ell+1)m}{6} & \text{if we are in case (1)}, \\
\frac{(\ell+1)m}{2} & \text{if we are in case (2)}.
\end{cases} \label{eq:m-ineq}
\end{align}
If we are in case (1), then $\ell \leq  \tfrac{1}{6} (\sqrt{24 n+1}-1)$. Similarly, in case (2), $P_{\ell+1}\le \tfrac{1}{6} (\sqrt{24 n+1}-5)$. Using these bounds on $\ell$ and \eqref{eq:m-ineq}, we find the lower bound in \eqref{eq:bound-on-m}, proving the theorem.
\end{proof}

\section{Proof of Theorem~\ref{thm:recursive-formula-for-u-k}}\label{sec:unimodal}
In this section we prove Theorem~\ref{thm:recursive-formula-for-u-k}.

\begin{proof}[Proof of Theorem~\ref{thm:recursive-formula-for-u-k}]
By \cite[Proposition~2.1]{KL} and \eqref{def:Jacobi-theta-function}, we have
\begin{align}\label{eq:U-as-T-C-h}
	U(\zeta;q) &= \frac{i\sin(\pi z)\eta(\tau)}{q^{\frac{1}{24}}\vartheta(z;\tau)}    T(2z;\tau) + h(\z;q) .
\end{align}
In particular, we have, using \eqref{eq:def-G-H} and the fact that $[\frac{\partial}{\partial z} \vartheta(z;\tau)]_{z=0} = -2\pi\eta^3(\tau)$,
\begin{align*}
	U(1;q) &= \frac{T_0(\tau)}{2iq^{\frac{1}{24}}\eta^2(\tau)}.
\end{align*}
Plugging this into \eqref{eq:U-as-exp} gives
\begin{align*}
	U(\zeta;q)&=\frac{\sin(\pi z) T_0(\tau)}{2\pi iz\eta^2(\tau)}\exp\lrb{ 2\sum_{k\ge 1} u_k(\tau) \frac{(2\pi i z)^k}{k!}}.
\end{align*}
Using this in \eqref{eq:U-as-T-C-h}, we obtain
\begin{align*}
	\frac{\sin(\pi z) T_0(\tau)}{2\pi iz\eta^2(\tau)} \exp \lrb{ 2\sum_{k\ge 1} u_k(\tau) \frac{(2\pi i z)^k}{k!}} &= \frac{i\sin(\pi z)\eta(\tau)}{q^{\frac{1}{24}}\vartheta(z;\tau)} T(2z;\tau) + h(\z;q).
\end{align*}
Equivalently, using \eqref{eq:vartheta-as-Eis}, \eqref{def:mathcalT-as-exp}, and \eqref{def:mathcalH-as-exp}, we have
\begin{multline*}
	 \exp\lrb{ 2\sum_{k\ge 1} u_k(\tau) \frac{(2\pi i z)^k}{k!}} \\
	= \exp\lrb{\sum_{k\ge 1} \left(2G_k(\tau)-2^k g_k(\tau)\right)\frac{(2\pi iz)^k}{k!}} + \frac{4\pi z \eta^3(\tau)}{T_0(\tau)}\exp\lrb{-\sum_{k\ge 1} h_k(\tau)\frac{(2\pi iz)^k}{k!}}.
\end{multline*}
Using this identity and Lemma~\ref{lem:cycle-index}, we obtain 
\begin{multline*}
	\sum_{{k\ge0}}\sum_{\lambda\vdash k}\prod_{j=1}^k\frac{\left(\frac{2u_j(\tau)}{j!}\right)^{m_j}}{m_j!}(2\pi iz)^k \\
	= \sum_{{k\ge0}}\sum_{\lambda\vdash k}\prod_{j=1}^k\frac{\left(\frac{2G_j(\tau)-2^j g_j(\tau)}{j!}\right)^{m_j}}{m_j!} (2\pi iz)^k +  \frac{4\pi z\eta^3(\tau)}{T_0(\tau)} \sum_{{k\ge0}}\sum_{\lambda\vdash k}\prod_{j=1}^k\frac{\left(\frac{-h_j(\tau)}{j!}\right)^{m_j}}{m_j!} (2\pi iz)^k.
\end{multline*}
In the language of partition traces, we have
\begin{align*}
	\sum_{k\ge 0}\mathrm{Tr}_k(\phi,u;\tau) (2\pi iz)^k = \sum_{k\ge 0} \mathrm{Tr}_k(\phi,\gamma;\tau) (2\pi iz)^k - \frac{2i \eta^3(\tau)}{T_0(\tau)} \sum_{k\ge 1} \mathrm{Tr}_{k-1}(\psi,h;\tau)(2\pi iz)^k.
\end{align*}
Separating the partition $(k^1)$ from the rest, this yields
\begin{align}
	\frac{2u_k(\tau)}{k!} &= -  \sum_{\substack{\lambda\vdash k\\ \lambda\neq \left(k^1\right)}} \phi(\lambda) u_\lambda(\tau)+ \mathrm{Tr}_k(\phi,\gamma;\tau) - \frac{2i\eta^3(\tau)}{T_0(\tau)} \mathrm{Tr}_{k-1}(\psi,h;\tau).\label{eq:u-n-as-trace-mid}
\end{align}
Using \eqref{def:mathcalT-as-exp}, we obtain that
\begin{align}\label{eq:g-1-as-eta-T-0}
	g_1(\tau) &= -\frac{1}{2\pi i} \frac{\left[\frac{\partial}{\partial z}T(z;\tau)\right]_{z=0}}{T_0(\tau)}=-\frac{i\eta^3(\tau)}{T_0(\tau)}.
\end{align}
Plugging this into \eqref{eq:u-n-as-trace-mid} gives \eqref{eq:u_k}.

Note that the components of $\gamma=\{G_k-2^{k-1}g_k\}_{k\in\N}$ and $h=\{h_k\}_{k\in\N}$ belong to $\widetilde{\mathcal{M}}[g_1,g_2,\ldots,h_1,$ $h_2,\ldots]$.
Moreover, since $z \mapsto U(e^{2\pi iz};q)$ is even, we have $u_k=0$ for $k$ odd by \eqref{eq:U-as-exp}. In particular $u_1=0$. Using this and \eqref{eq:u_k} inductively we have $u_k\in \widetilde{\mathcal{M}}[g_1,g_2,\ldots,h_1,h_2,\ldots].$ This space is closed under the action of $D$ by Theorems~\ref{thm:mathbbT-algebra} and~\ref{thm:h-n-diff}. 
Next we show that $\mathcal{U} = \widetilde{\mathcal{M}}[g_1,g_2,\ldots,h_1,h_2,\ldots]$. By definition  $\mathcal{U}\subseteq\widetilde{\mathcal{M}}[g_1,g_2,\ldots,h_1,h_2,\ldots]$. For the other containment, we need to show that $g_k\in\mathcal{U}$ for $k\ge 3$ odd. To see this observe, using \eqref{eq:g-1-as-eta-T-0}, \eqref{eq:der-eta-as-G-2}, and Theorem~\ref{thm:mathbbT-algebra}, that 
\begin{equation*}
	\frac{1}{g_1}D(g_1)=\frac{1}{\frac{\eta^3}{T_0}} D\lrb{\frac{\eta^3}{T_0}} = D\left(\Log\lrb{\frac{\eta^3}{T_0}}\right) = 3D(\Log(\eta))-D(\Log(T_0))=-3G_2 + \frac{g_2}{2} -\frac{g_1^2}{2}.
\end{equation*}
The claim then follows by induction, using Theorem~\ref{thm:mathbbT-algebra}. This completes the proof.
\end{proof}

\section{Examples} \label{sec:examples}
We have 
\begin{align*}
g_{1}(\tau) &= -\frac{1}{2} + q + q^{2} - q^{3} - 2 q^{4} - 3 q^{5} + q^{6} + 4 q^{7} + 8 q^{8} + O\!\left(q^{9}\right)\!, \\
g_{2}(\tau) &= q + 2 q^{2} - q^{3} - 5 q^{4} - 11 q^{5} - 2 q^{6} + 12 q^{7} + 38 q^{8} + O\!\left(q^{9}\right)\!, \\
g_{3}(\tau) &= q + 4 q^{2} + q^{3} - 11 q^{4} - 39 q^{5} - 30 q^{6} + 22 q^{7} + 170 q^{8} + O\!\left(q^{9}\right)\!, \\
g_{4}(\tau) &= q + 8 q^{2} + 11 q^{3} - 17 q^{4} - 131 q^{5} - 200 q^{6} - 72 q^{7} + 680 q^{8} + O\!\left(q^{9}\right)\!, \\
g_{5}(\tau) &= q + 16 q^{2} + 49 q^{3} + 13 q^{4} - 399 q^{5} - 1074 q^{6} - 1226 q^{7} + 2078 q^{8} + O\!\left(q^{9}\right) , \\
g_{6}(\tau) &= q + 32 q^{2} + 179 q^{3} + 295 q^{4} - 971 q^{5} - 5072 q^{6} - 10128 q^{7} + 728 q^{8} + O\!\left(q^{9}\right)
\intertext{and}
h_{1}(\tau) &= -\frac{1}{2} + 2 q + 5 q^{2} + 7 q^{3} + 12 q^{4} + 14 q^{5} + 24 q^{6} + 27 q^{7} + 42 q^{8} + O\!\left(q^{9}\right)\!, \\
h_{2}(\tau) &= 4 q + 17 q^{2} + 37 q^{3} + 83 q^{4} + 140 q^{5} + 273 q^{6} + 425 q^{7} + 736 q^{8} + O\!\left(q^{9}\right)\!, \\
h_{3}(\tau) &= 8 q + 59 q^{2} + 197 q^{3} + 579 q^{4} + 1316 q^{5} + 3019 q^{6} + 5919 q^{7} + 11730 q^{8} + O\!\left(q^{9}\right)\!, \\
h_{4}(\tau) &= 16 q + 209 q^{2} + 1057 q^{3} + 4073 q^{4} + 12032 q^{5} + 32883 q^{6} + 78209 q^{7} + 178426 q^{8} + O\!\left(q^{9}\right)\!, \\
h_{5}(\tau) &= 32 q + 755 q^{2} + 5717 q^{3} + 28887 q^{4} + 108692 q^{5} + 355399 q^{6} + 1007247 q^{7} + 2645982 q^{8} + O\!\left(q^{9}\right)\!, \\
h_{6}(\tau) &= 64 q + 2777 q^{2} + 31177 q^{3} + 206513 q^{4} + 977960 q^{5} + 3828723 q^{6} + 12805745 q^{7} + O\!\left(q^{8}\right)\!. 
\end{align*}

\section{Questions for future research} \label{sec:questions}
Here we raise some questions for future research.
\begin{enumerate}[leftmargin=*]
	\item By Proposition~\ref{prop:WidehatH-transform}, the function $\widehat{H}$ defined in \eqref{def:widehatMathcalH} transforms modular on $\Gamma_1(3)$. Note that it could also be viewed as a vector-valued Jacobi form on $\mathrm{SL}_2(\Z)$. It is thus natural to ask how the remaining components of the associated vector--valued Jacobi form can be described or characterized. Are they related to interesting combinatorial objects?
	\item Computational data suggests that the Fourier coefficients of $h_k$ are positive (except for $\operatorname{coeff}_{[q^0]}h_1$ $=-\tfrac{1}{2}$). This raises the question of whether these coefficients possess a combinatorial interpretation. Furthermore, the description of the Fourier coefficients in Theorem~\ref{thm:Fourier-exp-of-h_k} appears insufficient to establish positivity, since the quantities~$b_{n,m}$ are not always positive. Resolving this question might require a more refined analysis or a different approach.
	\item In \cite[equation (3.7)]{BKMN21}, analogues of the raising and lowering operators were introduced in the context of the completion of false objects. A natural question is to investigate their actions on $\widehat{T}$ and $\widehat{H}$, as well as their induced actions on $\widehat{g}_k$ and $\widehat{h}_k$.
	\item Can one utilize the raising and lowering operators mentioned in the previous question to study the action of heat operator on $\widehat{T}$ and $\widehat{H}$, or equivalently obtain a rank-crank type PDE for $\widehat{T}$ and $\widehat{H}$? This would enable us to understand the spaces generated by $\widehat{g}_k$ and $\widehat{h}_k$, respectively. 
\end{enumerate}

\end{document}